\documentclass[10 pt, a4paper,oneside]{amsart}
\usepackage{amsthm, cite, enumerate, hyperref, graphicx,subcaption}
\usepackage[all]{xy}

\usepackage{amssymb,xcolor}
\hypersetup{
  colorlinks,
  linkcolor={blue!50!black},
  citecolor={blue!50!black},
  urlcolor={blue!80!black}
}

\usepackage[inner=39.59mm,top=32.59mm,outer=39.59mm,bottom=52.74mm]{geometry}
\newenvironment{myabstract}{\list{}{\leftmargin=10mm\rightmargin=10mm}\item[]}{\endlist}

\graphicspath{{images/}}

\theoremstyle{plain}

\newtheorem{thm}{Theorem}[]

\newtheorem{lem}[thm]{Lemma}

\newtheorem{conjecture}{Conjecture}

\theoremstyle{definition}

\newtheorem{dfn}[thm]{Definition}

\DeclareMathOperator{\Defo}{Def}

\DeclareMathOperator{\Hom}{Hom}

\DeclareMathOperator{\hcf}{hcf}

\DeclareMathOperator{\Ker}{Ker}
\DeclareMathOperator{\Spec}{Spec}
\DeclareMathOperator{\Ext}{Ext}
\DeclareMathOperator{\shExt}{\underline{\mathit{Ext}}}
\DeclareMathOperator{\shHom}{\underline{\mathit{Hom}}}

\newcommand{\cB}{\mathcal{B}}

\newcommand{\bu}{\mathbf{u}}
\newcommand{\one}{\mathbf{1}}

\newcommand{\oo}{\mathcal{O}}
\newcommand{\QQ}{{\mathbb{Q}}}
\newcommand{\RR}{{\mathbb{R}}}
\newcommand{\PP}{{\mathbb{P}}}
\newcommand{\ZZ}{{\mathbb{Z}}}

\newcommand{\CC}{\mathbb{C}}

\newcommand{\TT}{\mathbb{T}}

\newcommand{\sX}{\mathfrak{X}}
\newcommand{\cX}{{\mathfrak{X}^\text{can}}}
\newcommand{\cXi}{{\mathfrak{X}^\text{can}_i}}
\newcommand{\Cstar}{\CC^\times}
\newcommand{\ve}{\varepsilon}

\newcommand{\hmin}{h_{\text{\it min}}}
\newcommand{\hmax}{h_{\text{\it max}}}
\begin{document}

\title[Mirror Symmetry and the Classification of Orbifold del~Pezzo Surfaces]
{Mirror Symmetry and the Classification of \\Orbifold del~Pezzo Surfaces}

\author[Akhtar]{Mohammad Akhtar}
\author[Coates]{Tom Coates}
\author[Corti]{Alessio Corti}
\author[Heuberger]{Liana Heuberger}
\author[Kasprzyk]{Alexander Kasprzyk}
\author[Oneto]{Alessandro Oneto}
\author[Petracci]{Andrea Petracci}
\author[Prince]{Thomas Prince}
\author[Tveiten]{Ketil Tveiten}

\address{Department of Mathematics\\Imperial College London\\180 Queen's Gate\\London, SW7 2AZ, UK} 
\email{mohammad.akhtar03@imperial.ac.uk}
\email{t.coates@imperial.ac.uk}
\email{a.corti@imperial.ac.uk}
\email{a.m.kasprzyk@imperial.ac.uk}
\email{a.petracci13@imperial.ac.uk}
\email{t.prince12@imperial.ac.uk}

\address{Institut Mathematique de Jussieu\\4 Place Jussieu\\75005 Paris, France}
\email{liana.heuberger@imj-prg.fr}

\address{Department of Mathematics, Stockholm University, SE-106 91, Stockholm, Sweden}
\email{oneto@math.su.se}
\email{ktveiten@math.su.se}

\thanks{Supported by EPSRC grant EP/I008128/1 and ERC Starting Investigator Grant 240123.}

\date{\today}

\subjclass[2000]{}

\maketitle

\begin{myabstract}
  % \vspace{-\baselineskip} 
  \footnotesize {\sc Abstract.}
  We state a number of conjectures that together allow one to classify a broad class of del~Pezzo surfaces with cyclic quotient singularities using mirror symmetry.  We prove our conjectures in the simplest cases.  The conjectures relate mutation-equivalence classes of Fano polygons with $\QQ$-Gorenstein deformation classes of del~Pezzo surfaces.\\ \\
\end{myabstract}

\fontsize{10pt}{12.5pt}\selectfont

\noindent We explore mirror symmetry for del~Pezzo surfaces with cyclic quotient singularities. We begin by stating two logically independent conjectures.  In Conjecture~\ref{conjecture:A} we try to imagine what consequences mirror symmetry may have for classification theory.  In Conjecture~\ref{conjecture:B} we make what we mean by mirror symmetry precise.  
This work owes a great deal to conversations with Sergey Galkin, and to the pioneering papers by Galkin--Usnich~\cite{GU} and Gross--Hacking--Keel~\cite{GHK11, GHK12}.

\section*{Basic Concepts}

\noindent Consider a del~Pezzo surface $X$ with isolated cyclic
quotient singularities. $X$ is analytically locally (or \'etale
locally if you prefer) isomorphic to a quotient $\CC^2/\mu_n$, where
without loss of generality $\mu_n$ acts with weights $(1,q)$ with
$\hcf (q,n)=1$. We denote the quotient\footnote{We think of quotient singularities
  themselves as either analytic germs or formal algebraic germs
  $(x,X)$.} of $\CC^2$ by this action by
$\frac{1}{n}(1,q)$. There is a canonical way to regard $X$ as a non-singular
Deligne--Mumford stack with non-trivial isotropy only at isolated
points; we will denote this stack by $\sX$, writing $X$ for the
underlying variety.  The canonical class of $X$ is a $\QQ$-Cartier
divisor and thus it makes sense to say that $X$ is a \emph{del~Pezzo
  surface}, that is, that the anti-canonical divisor $-K_X$ is ample.

There is a notion of \emph{$\QQ$-Gorenstein (qG) deformation} of varieties with quotient
singularities, and  of miniversal
qG-deformation~\cite{MR922803, MR1144527} . The smallest positive integer $r$ such that $rK_X$ is
Cartier is called the \emph{Gorenstein index}. If $S$ is the spectrum
of a local Artin ring, the key defining properties of a qG-deformation
$f \colon \mathcal{X}\to S$ of $(x,X)$ are flatness and that
$rK_{\mathcal{X}/S}$ be a relative Cartier divisor, where
$K_{\mathcal{X}/S}$ is the relative canonical class. Thus, for qG-deformations, the invariant $K_X^2$ of fibres is
locally constant on the base, and hence, for a qG-deformation of a del~Pezzo surface,
$h^0(X,-K_X)$ of fibres is also locally constant on the base. For a
quotient singularity $\frac{1}{n}(1,q)$ write $q=p-1$, $w=\hcf(n,p)$, $n=wr$,
$p=wa$; then $r$ is the Gorenstein index and we call $w$ the
\emph{width} of the
singularity~\cite{AK14}. It is easy to see that $\frac1{n}(1,q)$ is
\[
(xy+z^w=0)\subset \textstyle \frac1{r}(1,wa-1,a)
\]
where $x$, $y$, $z$ are the standard co-ordinate functions on
$\CC^3$.  Write $w=mr+w_0$ with $0\leq w_0<r$. It is known~\cite{MR922803, MR1144527} that the base of the
miniversal qG-deformation\footnote{The moduli space of arbitrary flat deformations of $\frac1{n}(1,q)$ has many components.  Little is known about these components in general, but the distinguished component corresponding to qG-deformations is smooth and reduced.} of $\frac1{n}(1,q)$ is isomorphic to $\CC^{m-1}$
and, choosing co-ordinate functions $a_1,\dots,a_{m-1}$ on it, the miniversal qG-family is
given explicitly by the equation:
\[
\bigl(xy+(z^{rm}+a_1z^{r(m-2)}+\cdots + a_{m-1})z^{w_0}=0\bigr)\subset
\textstyle \frac1{r}(1,w_0a-1,a)\times \CC^{m-1} 
\]

We say that $\frac1{n}(1,q)$ is of \emph{class T} or is a \emph{$T$-singularity} if $w_0=0$, and
that it is a \emph{primitive $T$-singularity} if $w_0=0$ and $m=1$.  $T$-singularities appear in the work of Wahl~\cite{MR597833} and Koll\'ar--Shepherd-Barron~\cite{MR922803}.  We say that $\frac1{n}(1,q)$ is of \emph{class R} or is a \emph{residual singularity} if $m=0$, that is, if $w=w_0$.  We say that the singularity 
\[
\textstyle \frac1{w_0r}(1,w_0a-1)=(xy+z^{w_0}=0)\subset \frac1{r}(1,w_0a-1,a)
\] 
is the \emph{R-content} of $\frac1{n}(1,q)$ and that the pair
$\bigl(m,\frac1{w_0r}(1,w_0a-1)\bigr)$ of a non-negative integer and a
singularity is the \emph{singularity content} of
$\frac1{n}(1,q)$.  Residual singularities and singularity content appear in the work of Akhtar--Kasprzyk~\cite{AK14}.  The generic fibre of the miniversal family of $\frac1{n}(1,q)$ has a
unique singularity of class R, the R-content, and a singularity is
\emph{qG-rigid} if and only if it is of class R. At the opposite end of
the spectrum, a singularity is of class T if and only if it admits a
qG-smoothing.

In our formulation below, one side of mirror symmetry consists of the
set of qG-deformation classes of locally qG-rigid del~Pezzo surfaces,
that is, of del~Pezzo surfaces with residual singularities. In order to
make sense of the other side of mirror symmetry, we need to discuss
mutations of Fano polygons.  Fix a lattice $N\cong \ZZ^d$ and its dual lattice $M=\Hom(N, \ZZ)$.  
A \emph{Fano polytope} is a convex lattice polytope $P\subset N_\RR$ such that:
\begin{enumerate}
\item[1.] the origin $0\in N$ lies in the strict interior of $P$;
\item[2.] the vertices $\rho_i\in N$ of $P$ are primitive lattice vectors.
\end{enumerate}
For a Fano polygon $P$ we
denote by $X_P$ the toric variety defined by the spanning fan of $P$; this is a del~Pezzo surface with cyclic quotient
singularities. There is a notion of \emph{combinatorial mutation}~\cite{MR3007265} of lattice polytopes, which we now describe in the special case of lattice polygons.  Let $P\subset N$ be a lattice polygon.  
\emph{Mutation data} \label{def:mutation} for $P$ is the choice of primitive\footnote{In the original work~\cite{MR3007265}, the vector $f$ was not required to be primitive.  Any combinatorial mutation in the original sense can be written as a composition of mutations with primitive $f$.} vectors $h\in
M$ and $f\in h^\perp \subset N$ satisfying the following two
conditions. Denote by $\hmax>0$ and $\hmin<0$ the maximum and minimum
values of $h$ on $P$. Choose an orientation of $N$
and label the vertices of $P$ by $\rho_1, \rho_2, \dots$
counterclockwise, such that $h(\rho_1)=\hmax$. The conditions are:
\begin{itemize}
\item there is an edge $E_i=[\rho_i,\rho_{i+1}]$ such that $h(\rho_i) = h(\rho_{i+1}) = \hmin$; 
\item $\rho_{i+1}-\rho_i=wf$ where $w\geq {-\hmin}$ is an integer. 
\end{itemize}
Informally, to mutate $P$ we just add $kf$ at height $k\geq 0$, and
take away $-kf$ at height $k<0$. The conditions on the mutation data
simply mean that it is possible to take away $-kf$ at height $k<0$.
In describing precisely the construction of the mutation of $P$ we
distinguish two cases:
\begin{enumerate}[(I)]
\item[I.] $P$ has $m$ vertices, $\rho_1, \dots, \rho_m$, and $\rho_1$ is the
  unique maximum for $h$ on $P$;
\item[II.] $P$ has $m+1$ vertices $\rho_1, \dots, \rho_{m+1}$, and
  $h(\rho_1)=h(\rho_{m+1})=\hmax$. 
\end{enumerate}
The \emph{mutation} of $P$ with respect to the mutation data $(h,f)$ is the Fano polygon $P^\prime$ with vertices:
\begin{align*}
  & \rho^\prime_j=
    \begin{cases}
      \rho_j & \;\text{if $1\leq j \leq i$} \\
      \rho_j+h(\rho_j) f & \;\text{if $i < j \leq m$} \\
      \rho_1+\hmax f & \;\text{if $j = m+1$}
    \end{cases}
                       \intertext{in case~I, and}
  & \rho^\prime_j=
    \begin{cases}
      \rho_j & \;\text{if $1\leq j \leq i$} \\
      \rho_j+h(\rho_j) f & \;\text{if $i < j \leq m$} \\
      \rho_{m+1}+\hmax f & \;\text{if $j = m+1$}
    \end{cases}
\end{align*}
in case~II.

The definition of mutation becomes more transparent if we consider $Q \subset M$, the \label{def:dual_mutation} polygon dual to $P$.  Let $\psi\colon M \to M$ be the piecewise-linear map defined by:
\[
\psi (u) = u-\min\big({\langle f,u \rangle},0 \big) \, h
\]
If $Q^\prime$ denotes the dual to the mutated polygon $P^\prime$,
then $Q^\prime =\psi (Q)$.

\section*{Conjecture A}

\begin{dfn}
  \label{dfn:4}
  A del~Pezzo surface with cyclic quotient singularities is of class
  TG (for Toric Generization) if it admits a
  qG-degeneration with reduced fibres to a normal toric del~Pezzo surface. 
\end{dfn}

Not all locally qG-rigid del~Pezzo surfaces with cyclic quotient
singularities are of class TG. Consider, for example, the complete
intersection $X_{6,6}\subset \PP(2,2,3,3,3)$.  This surface has $4$
singularities of type $\frac1{3}(1,1)$, and degree $K_X^2=\frac1{3}$; it is not of class
TG because $h^0(X,-K_X)=h^0\bigl(X,\oo_X(1)\bigr)=0$. It is
an open and apparently difficult question to give a meaningful
characterization of surfaces of class TG.  

\begin{dfn}
  \label{dfn:equiv}
  Fano polygons $P$, $P^\prime$ are \emph{mutation
    equivalent} if there is a sequence of combinatorial mutations that starts from $P$ and
  ends at $P^\prime$.  Del~Pezzo surfaces $X$, $X^\prime$ with cyclic quotient
  singularities are \emph{qG-deformation equivalent} if there exist
  qG-families $f_i\colon \mathcal{X}_i\to S_i$ over connected schemes $S_i$, $1 \leq i \leq n$, and points $t_i$,~$s_i\in S_i$ such that we have the following equalities of scheme-theoretic inverse images:
  \begin{align*}
    X=f_1^* (t_1) && 
                         \text{$f_i^* (s_i)=f_{i+1}^* (t_{i+1})$ for $1 \leq i < n$} &&
                                                                                                  f_{n}^* (s_n)=X^\prime
  \end{align*}
\end{dfn}

\noindent Lemma~\ref{lem:1} below states, in particular, that qG-deformations
of del~Pezzo surfaces with cyclic quotient singularities are
unobstructed. Thus it would suffice to take $n=1$ in
Definition~\ref{dfn:equiv}.

\begin{conjecture} \label{conjecture:A} There is a one-to-one correspondence between:
  \begin{itemize}
  \item the set $\mathfrak{P}$ of mutation equivalence classes of Fano
     polygons; and
  \item the set $\mathfrak{F}$ of qG-deformation equivalence classes
    of locally qG-rigid class~TG del~Pezzo surfaces with cyclic quotient
    singularities.
  \end{itemize}
  The correspondence sends $P$ to a (any) generic qG-deformation of
  the toric surface $X_P$.
\end{conjecture}

\noindent We will prove half of Conjecture~\ref{conjecture:A} below:

\begin{thm}
  \label{thm:1}
  The assignment, to a Fano polygon $P$, of a (any) generic qG-deformation 
  of the toric surface $X_P$ defines a surjective map
  $\mathfrak{P} \to \mathfrak{F}$. 
\end{thm}

 The real content of Conjecture~A is the statement that the map
$\mathfrak{P} \to \mathfrak{F}$ is injective. This is a strong
statement about the structure of the boundary of the stack of del~Pezzo surfaces. In Lemma~\ref{lem:2} below, we attach to a mutation between
Fano polygons $P$ and $P^\prime$ a special qG-pencil $g\colon \mathcal{X} \to
\PP^1$ with scheme-theoretic fibres $g^* (0) = X_P$ and $g^* (\infty)
= X_{P^\prime}$. By construction all fibres of $g$ come with an action
of $\CC^\times$; indeed they are $T$-varieties in the sense of Altmann~\emph{et al.}~\cite{AIPSV,AH,AHS}.
Conjecture~A states that, if the toric surfaces $X_P$
and $X_{P^\prime}$ are deformation equivalent, then the corresponding
points in the moduli stack are connected by a chain of $\PP^1$s
given by such special qG-pencils.

\section*{Conjecture B}

\noindent Let $P$ be a Fano polygon and $X$ a generic qG-deformation of
the surface $X_P$. The second of our two conjectures relates the
quantum cohomology of $X$ to the variation of homology of fibres of
certain Laurent polynomials with Newton polygon $P$. We introduce the
key ingredients that we need in order to state it. We begin by describing the
quantum cohomology side.  

The surface $X$ is a del~Pezzo surface with cyclic quotient
singularities.  Denote the singularities by
$(x_j, X)\cong \frac1{n_j}(1,q_j)$, $j \in J$, where $J$ is an index
set.  Let $\sX$ denote the surface $X$ but regarded as a smooth
Deligne--Mumford stack with isotropy only at the points
$x_j$,~$j \in J$.  Let $H_X$ denote the Chen--Ruan orbifold cohomology
of $\sX$, that is, the cohomology of the inertia stack $I \sX$ with
shifted grading.  As a vector space, we have:
\begin{align*}
  H_X=\left(\bigoplus_kH^{2k} (X;\CC) \right) \oplus \left(\bigoplus_{j \in J} H^{\text{tw}}_{x_j}\right)  
  && \text{where}
  && H^{\text{tw}}_{x_j} = \bigoplus_i \CC \one_{i,j} 
\end{align*}
and the index $i$ in the definition of the `twisted sector'
$H^{\text{tw}}_{x_j}$ runs over the set of non-zero elements in
$\frac1{n_j}\ZZ/\ZZ$.  The element $ \one_{i,j}$ has degree
$\big\{\frac{i}{n_j}\big\} + \big\{\frac{i q_j}{n_j}\big\}$, where
$\{x\}$ denotes the fractional part of the rational number $x$, and
elements of $H^{2k} (X;\CC) \subset H_X$ have degree $k$.

Given $\alpha_1,\ldots,\alpha_n \in H_X$, non-negative integers
$k_1,\ldots,k_n$, and $\beta \in H_2(X;\QQ)$, one can consider the
genus-zero Gromov--Witten invariant of $\sX$:
\[
\bigl\langle \alpha_1 \psi_1^{k_1},\ldots,\alpha_n\psi_n^{k_n}
\bigr\rangle_{0,n,\beta}
\]
This is defined in~\cite{MR1950941, MR2172496, MR2450211}; roughly
speaking, it counts the number of genus-zero degree-$\beta$ orbifold
curves in $\sX$, passing through various cycles in $\sX$ and with
isotropy specified by $\alpha_1,\ldots,\alpha_n$.  Denoting by
$\bu_1, \dots, \bu_s$ those classes $\one_{i,j}$ with
$0<\deg \one_{i,j}< 1$ in some order, the \emph{quantum period} of
$\sX$ is the power series:
\[
G_\sX (x,q)=\sum_{\beta \in H_2(X;\QQ)} \sum_{n=0}^\infty
\sum_{1\leq i_1, \dots ,i_n\leq s}
\Bigg\langle \bu_{i_1},\ldots, \bu_{i_n},
\frac{[\text{pt}]}{1-\psi_{n+1}}  \Bigg\rangle_{0,n+1,\beta} 
\frac{x_{i_1}\cdots x_{i_n}}{n!} \, q^\beta
\]
Composing with the substitution
$q^\beta \mapsto t^{-K_\sX\cdot \beta}$,
$x_i \mapsto x_i t^{1- \deg \bu_i}$ defines a formal power
series\footnote{The formula for the virtual dimension of the moduli
  space of stable maps to $\sX$~\cite{MR2104605} ensures that the
  powers of $t$ occurring in $G_\sX$ are integral.  In this context both $G_\sX(x, t)$ and $\widehat{G}_\sX (x,t)$ are
  elements of $\QQ[x_1,\ldots, x_s][\![t]\!]$; see \cite{OP14} for details.}:
\begin{align*}
  G_\sX (x,t)&=\sum_{d=0}^\infty  c_d(x)t^d \\
  \intertext{The \emph{regularized quantum period} of $\sX$
  is:}
  \widehat{G}_\sX(x, t)& =\sum_{d=0}^\infty d!\,c_d(x)\,t^d
\end{align*}

This concludes our description of the quantum cohomology side of
Conjecture~\ref{conjecture:B}; we now describe the other side. We consider Laurent polynomials
\[
g=\sum_{\gamma \in N\cap P} a_\gamma x^\gamma
\]
with Newton polygon equal to the Fano polygon $P$.  

Let $h\in M$ and $f\in h^\perp\subset N$ be mutation data for $P$. The \emph{cluster transformation}
\[
\Phi \colon x^\gamma \mapsto x^\gamma (1+x^f)^{\langle \gamma, h\rangle}
\]
defines an automorphism of the field of fractions
$\CC(N)$ of $\CC[N]$, and we say that the Laurent polynomial $g\in \CC[N]$
is \emph{mutable} with respect to $(h,f)$ if $g\circ \Phi$ lies in $\CC[N]$.  It is easy to see that if $g$ is mutable then the Newton polygon of $g^\prime := g\circ \Phi$ is
the mutated polygon $P^\prime$.

\begin{dfn}
  Let $P$ be a Fano polygon\footnote{Kasprzyk--Tveiten have defined
    the correct notion of maximal-mutability for Laurent polynomials
    in more than two variables: see~\cite{KT14}.  The
    many-variables case presents many new features.} and let $g \in \CC[N]$,
  \[
  g=\sum_{\gamma \in N\cap P} a_\gamma x^\gamma
  \]
  be a Laurent polynomial with Newton polygon $P$.  We say that $g$ is \emph{maximally-mutable} if:
  \begin{itemize}
  \item for each positive integer $n$ and each sequence of mutations
    \[
    P_0 \to P_1 \to P_2 \to \cdots \to P_n
    \]
    with $P_0 = P$, there exist Laurent polynomials $g_i \in \CC[N]$
    with $g_0 = g$ such that the Newton polygon of $g_i$ is $P_i$ and
    the cluster transformation $\Phi_i$ determined by the mutation
    $P_i \to P_{i+1}$ satisfies $g_i \circ \Phi_i = g_{i+1}$.
  \item $a_0=0$; this is just a convenient normalization condition.
  \end{itemize}
\end{dfn}

\noindent The set of maximally-mutable Laurent polynomials with Newton polygon
$P$ is a vector space over $\CC$ that we denote by $L_P$.

We say that the Laurent polynomial $g$ has \emph{$T$-binomial edge coefficients} if successive coefficients $a_\gamma$ along each edge of $P$ of height $r$ and width $w$, where $w = mr+w_0$ with $0 \leq w_0<r$, are successive coefficients of $T$ in
\[
\begin{cases}
  (1+T)^{mr} & \text{if $w_0=0$} \\
  (1+T)^{mr} (1+T^{w_0}) & \text{if $w_0\ne 0$.} \\
\end{cases}
\]
If $g$ has $T$-binomial edge coefficients and $\rho$ is a vertex of
$P$ then the coefficient $a_\rho = 1$.  If $X_P$ has only
$T$-singularities (that is, in the language of
Definition~\ref{dfn:singularity_content} below, if the basket $\cB$ of
$P$ is empty) then $T$-binomial edge coefficients are binomial
coefficients.

\noindent Kasprzyk--Tveiten have shown that, for any Fano polygon $P$,
the set of maximally-mutable Laurent polynomials with Newton polygon
$P$~\cite{KT14} and $T$-binomial edge coefficients is an affine
subspace of $L_P$ that we denote by $L_P^T$.

There is a universal maximally-mutable Laurent polynomial:
\begin{equation}
  \label{eq:universal}
  \begin{aligned}
    \xymatrix{ L_P^T\times \TT \ar[r]^-g\ar[d]_{\text{pr}_1} & \CC \\
      L_P^T &}
  \end{aligned}
\end{equation}
where $\TT=\Spec \CC[N]$,  which we consider to be the
\emph{Landau--Ginzburg model}\footnote{More accurately, \eqref{eq:universal} is a torus chart on the
  Landau--Ginzburg mirror to $X$. One can use cluster transformations
  to glue different copies of $\TT$ to form a variety $Y$, and use the corresponding mutations to identify the different affine spaces $L_P^T \cong L_{P'}^T$. The
  maximally-mutable Laurent polynomials then define a global function
  $G\colon L_P^T\times Y \to \CC$. We will not pursue this here.}  mirror to a generic qG-deformation $X$ of the
surface $X_P$.
The \emph{classical period} of $P$ is the function of $a \in L_P^T$ and $t \in \CC$ defined by
\[
\pi_P (a,t)= \oint_{|x_1|=|x_2|=1} \frac1{1-tg(a,x)} \, \Omega
\]
where $\Omega$ is the invariant volume form on $\TT$ normalized such
that $ \oint_{|x_1|=|x_2|=1} \Omega=1$.

\begin{conjecture} \label{conjecture:B} Let $P$ be a Fano polygon and
  let $X$ be a generic qG-deformation of the toric surface $X_P$. Let
  $L_P^T$ denote the affine space of maximally-mutable Laurent
  polynomials with Newton polygon $P$ and $T$-binomial edge
  coefficients, and let $H_X^{\text{ts}} \subset H_X$ denote the
  twisted sectors of age less than~$1$:
  \[
  H_X^{\text{ts}}=\bigoplus_{i=1}^{r} \CC \bu_i
  \] 
  There is an affine-linear isomorphism
  $\varphi \colon L_P^T \to H_X^{\text{ts}}$, the mirror map, such
  that the regularized quantum period $\widehat{G}_\sX$ of $\sX$ and
  the classical period $\pi_P$ of $P$ satisfy\footnote{We think of
    $\widehat{G}_\sX$ and $\pi_P$ as functions from $H_X^{\text{ts}}$
    and $L_P^T$ to $\CC[\![t]\!]$.}
  $\widehat{G}_\sX\circ \varphi = \pi_P$.
\end{conjecture}

\noindent This Conjecture makes explicit an insight by Sergey Galkin, who several years ago suggested to us that mutable Laurent polynomials play a fundamental role in mirror symmetry.  

One might try to extend the subspace $H^{\text{\it ts}}_X \subset X$ to include classes of degree~$1$ from the twisted sectors and, correspondingly, to consider maximally-mutable Laurent polynomials with general (rather than $T$-binomial) edge coefficients.  One can formulate a version of Conjecture~\ref{conjecture:B} in this setting but in this case the mirror map $\varphi$ will in general no longer be affine-linear, being defined by a power series with finite radius of convergence.  One can see this already in the case of $X = \PP(1,1,6)$, where the quantum period can be computed using the Mirror Theorem for toric Deligne--Mumford stacks~\cite{CCIT,CCFK}, and the corresponding maximally-mutable Laurent polynomial is $f = x+y+x^{-1} y^{-6} + a_1 y^{-1} + a_2 y^{-2} + a_3 y^{-3}$ where $a_1$,~$a_2$, and~$a_3$ are parameters.

\section*{Two Further Conjectures}

\noindent We complete the picture by stating two further conjectures.

\begin{dfn}[\!\!\cite{AK14}] \label{dfn:singularity_content}
Let $P$ be a Fano polygon and denote the singular points of $X_P$ by $x_j$, $j \in J$.  Let $\big(m_j,\frac{1}{w_{0,j}r_j}(1,a_j w_{0,j}-1)\big)$ be the singularity content of $(x_j,X_P)$.  The \emph{singularity content} of $P$ is the pair $\bigr(m,\cB\bigl)$ where $m=\sum m_j$ and the multiset\footnote{In the original work by Akhtar--Kasprzyk $\cB$ is taken to be a cyclically ordered list, but the cyclic order will be unimportant in what follows.}
\[
\cB=\Bigl\{\textstyle\frac{1}{w_{0,j}r_j}(1,a_j w_{0,j}-1) : \text{$j \in J$, $w_{0,j} r_j \ne 1$}\Bigr\}
\]
is the basket of residual singularities of $X_P$. 
\end{dfn}

\noindent The singularity content of $P$ has an equivalent, purely combinatorial definition which we will not give here.  Akhtar--Kasprzyk have shown that the singularity content of $P$ is invariant under mutation.

\begin{conjecture} \label{conjecture:C} Let $P_1$ and $P_2$ be Fano
  polygons with the same singularity content.  Suppose that there is
  an affine-linear isomorphism $\varphi \colon L^T_{P_1} \to L^T_{P_2}$
  such that $\pi_{P_1}(a,t) = \pi_{P_2} (\varphi(a),t)$.  Then $P_2$
  is obtained from $P_1$ by a chain of mutations.
\end{conjecture}

\begin{conjecture} \label{conjecture:D}
  Let $X_1$ and $X_2$ be del~Pezzo surfaces of class TG with the same set of qG-rigid cyclic quotient singularities, and let $\varphi \colon H_{X_1}^{\text{ts}} \to H_{X_2}^{\text{ts}}$ be the obvious identification.  Suppose that $\widehat{G}_{\sX_1} = \widehat{G}_{\sX_2} \circ \varphi$.  Then $X_1$ and $X_2$ are qG-deformation equivalent.
\end{conjecture}

\noindent Conjectures~\ref{conjecture:B} and~\ref{conjecture:C}
together imply Conjectures~\ref{conjecture:A} and~\ref{conjecture:D}.
It would be very interesting to know whether
Conjectures~\ref{conjecture:A}, \ref{conjecture:B}
and~\ref{conjecture:D} together imply Conjecture~\ref{conjecture:C}.

\section*{The Proof of Theorem~\ref{thm:1}}
\label{sec:proofs}

\noindent We now prove Theorem~\ref{thm:1}, that is, we prove one half of Conjecture~\ref{conjecture:A}.  We begin with a result on
qG-deformations of del~Pezzo surfaces with cyclic quotient
singularities.

\begin{lem}
  \label{lem:1}
  Let $X$ be a del~Pezzo surface with cyclic quotient
  singularities $(x_i\in X)$. Then qG-deformations of $X$ are unobstructed and,
  denoting by $\Defo_{qG}X$ and $\Defo_{qG}(x_i, X)$ the global and local
  deformation functors, the morphism
\[
\Defo_{qG} X\to \prod_i \Defo_{qG} (x_i,X)
\]
 is formally smooth.
\end{lem}

\begin{proof} As before, let $(x_i, X)\cong 1/n_i(1,q_i)$ and write
  $q_i=p_i-1$, $w_i=\hcf (n_i,p_i)$, $n_i=w_ir_i$, and $p_i=w_ia_i$. Then
  $r_i$ is the local Gorenstein index at $x_i$ and the surface
  $Y_i$ given by the equation $(xy+z^{w_i}=0)$ in $\CC^3$ (with coordinates $x$,~$y$,~$z$) is the
  local (in the analytic or \'etale topology) canonical cover of
  $(x_i, X)$. Denote by $\cX$ the orbifold with local charts at $x_i$
  given by $\cXi =[Y_i/\mu_{r_i}]$
  at $x_i$. Then the qG-deformation functor of $X$ is the ordinary
  deformation functor of the orbifold $\cX$. Thus we work with the
  ordinary global and local deformation functors $\Defo \cX$, $\Defo
  (x_i, \cXi)$. The functor $\Defo \cX$ is controlled
  by $T^i=\Ext^i(\Omega^1_{\cX}, \oo_{\cX})$ in the standard way,
  and similarly for $\Defo (x_i, \cXi)$. Furthermore for our local
  models $\shExt^1(\Omega^1_{\cXi}, \oo_{\cXi})$ is a skyscraper
  sheaf supported at the singular point with fibre $\CC^{m_i-1}$, and
  all higher $\shExt^i$ vanish. We need to show that
  $\Ext^2(\Omega^1_{\cX}, \oo_{\cX})=0$ and that the natural map
\[
\Ext^1(\Omega^1_\cX, \oo_\cX) \to H^0\bigl(\cX, \shExt^1(\Omega^1_\cX,
\oo_\cX)\bigr)=\bigoplus_i \Ext^1(\Omega^1_{\cXi}, \oo_{\cXi})
\]
is surjective. As we explain in more detail below, this follows
easily from known vanishing theorems and the edge-sequence
of the local-to-global spectral sequence for computing $\Ext$ groups,
where as usual we denote by $\theta_\cX=\shHom (\Omega^1_{\cX},
\oo_\cX)$ the sheaf of derivations of $\cX$:
  \begin{multline*}
 H^1(\cX, \theta_\cX)\to \Ext^1(\Omega^1_\cX,\oo_\cX)\to
H^0\bigl(\cX, \shExt^1(\Omega^1_\cX, \oo_\cX)\bigr)\to \\ \to H^2(\cX,
\theta_\cX)\to \Ext^2_\cX(\Omega^1\cX,\oo_\cX)\to (0)    
  \end{multline*}
  (The last homomorphism here is surjective since all other groups on the
  $E_2$-page of the spectral sequence vanish.) Everything
  follows once we have established that $H^2(\cX, \theta_\cX)=(0)$. Indeed,
  let $\pi\colon \cX\to X$ be the forgetful morphism from the orbifold
  $\cX$ to its coarse moduli space $X$. It is obvious that, for every
  coherent sheaf $\mathcal{F}$ on $\cX$, $H^i(\cX,
  \mathcal{F})=H^i(X,\pi_* \mathcal{F})$. Now $\pi_*
  \theta_\cX$ is a torsion-free sheaf, hence we have an inclusion of
  sheaves
\[
\pi_* \theta_\cX \subset \bigl(\Omega^{1\, \vee \vee}_\cX\otimes
(-K_X)  \bigr)^{\vee \vee}
\]
as the sheaf on the right is saturated and the two sheaves coincide on
the smooth locus of $X$. So everything follows from vanishing of
$H^2\bigl(X,\bigl(\Omega^{1\, \vee \vee}_\cX\otimes (-K_X)  \bigr)^{\vee
  \vee} \bigr)$. But this group is Serre-dual to 
\begin{multline*}
  \Hom \Bigl(\bigl(\Omega^{1\, \vee \vee}_\cX\otimes
  (-K_X)  \bigr)^{\vee \vee},
  K_X\Bigr)
  =\Hom\Bigl(-K_X,\bigl(\theta_X^{\vee\vee}\otimes (K_X)^{\vee
  \vee} \Bigr)\\ =\Hom \bigl(-K_X, \Omega^{1\,\vee\vee}_X \bigr)=(0)
\end{multline*}
where vanishing of the last group follows from the Bogomolov--Sommese
vanishing theorem for varieties with log canonical singularities (see~\cite[7.2]{MR2854859} or~\cite[8.3]{MR2581247}).
 \end{proof}

\begin{lem}
  \label{lem:2}
  Let $P$ be a Fano polygon, let $(h,f)$ be mutation data for
  $P$, and let $P^\prime$ be the mutated polygon. There is a qG-pencil
  $g\colon \mathcal{X} \to \PP^1$ with scheme-theoretic fibres $g^* (0) = X_P$ and
  $g^* (\infty) = X_{P^\prime}$.
\end{lem}

\noindent  Without the conclusion that the pencil is qG, this statement was proved by Ilten~\cite{MR2958983}.  

\begin{proof}[Proof of Lemma~\ref{lem:2}]
 Let $\widetilde{M}=M\oplus \ZZ$ and denote elements $\tilde{u} \in \widetilde{M}$ by $(u,z) \in M \oplus \ZZ$. Let $\pi\colon \widetilde{M}\to M$ be the 
projection to the first factor and define $\pi^\prime \colon
\widetilde{M}\to M$ by $\pi^\prime (u, z) = u+zh$.  We will construct by explicit inequalities a convex rational polytope
$\widetilde{Q}\subset \widetilde{M}_\RR$ such that $\pi
(\widetilde{Q})=Q$ and $\pi^\prime (\widetilde{Q})=Q^\prime$, where $Q$ (respectively $Q'$) is the polygon dual to $P$ (respectively to $P'$). Denoting by
$\widetilde{X}$ the toric variety defined by the normal fan of
$\widetilde{Q}$, this gives embeddings $X_P\subset \widetilde{X}$ and
$X_{P^\prime}\subset \widetilde{X}$. We will conclude the proof by
writing an explicit homogeneous trinomial
\begin{equation}
  \label{eq:1}
xy+Az^wt^{w^\prime-r^\prime}+Bz^{w-r}t^{w^\prime}  
\end{equation}
in Cox coordinates for $\widetilde{X}$ such that
\begin{align}
  \label{eq:2}
X_P=\{xy+Az^wt^{w^\prime-r^\prime}=0\}
&& \text{and} && 
X_{P^\prime}=\{xy+Bz^{w-r}t^{w^\prime}=0\}  
\end{align}
and checking explicitly that it gives the desired qG-deformations. 

 Denote by $v_j\in Q$ the vertex corresponding to the edge $[\rho_j,
 \rho_{j+1}]\subset P$, and let $E_i = [\rho_i,\rho_{i+1}]$ be as in the definition of mutation (page~\pageref{def:mutation}).  Let $J = \{1,2,\ldots,m\} \setminus \{1,i,i+1\}$. Consider the following elements of
 $\widetilde{N}=N\oplus \ZZ$:  
 \begin{align*}
   \tilde{\rho}_x & = (f,1) \\
   \tilde{\rho}_y & = (0,1)\\
   \tilde{\rho}_z & = \left(\rho_i,\frac{1+\langle
     \rho_i,v_{i+1}\rangle}{\langle f, v_{i+1} \rangle} \right)=(\rho_i, -w)\\
   \tilde{\rho}_t & = \left(\rho_1,\frac{1+\langle
     \rho_1,v_{m}\rangle}{\langle f, v_{m} \rangle} \right)=(\rho_1,
   -w^\prime+r^\prime)\\ 
   \tilde{\rho}_j & = 
   \begin{cases}
     (\rho_j, 0) &\text{if $\langle \rho_j,h\rangle \geq 0$} \\
     (\rho^\prime_j, \langle \rho_j, h\rangle) &\text{if $\langle \rho_j,h\rangle < 0$}
   \end{cases}
                  &&
                     \text{for $j \in J$}
 \end{align*}
and let $\widetilde{Q}\subset \widetilde{M}_\QQ$ be the rational
polytope consisting of those $\tilde{u}\in \widetilde{M}$ that satisfy
the inequalities $\langle \tilde{\rho}_x, \tilde{u}\rangle \geq 0$, $\langle \tilde{\rho}_y, \tilde{u}\rangle \geq 0$, $\langle \tilde{\rho}_z, \tilde{u}\rangle \geq -1$, $\langle \tilde{\rho}_t, \tilde{u}\rangle \geq -1$, 
and $\langle \tilde{\rho}_j , \tilde{u}\rangle \geq -1$ for $j\in J$.  Let $\widetilde{X}$ be the toric variety defined by the normal fan of
$\widetilde{Q}$ and denote the corresponding Cox co-ordinates by $x$,~$y$,~$z$,~$t$,~$a_j$ for $j \in J$. It is
essentially immediate from the definition that $\pi(\widetilde{Q})=Q$
and $\pi^\prime (\widetilde{Q})=Q^\prime$. Consider the trinomial in
\eqref{eq:1} where:
\begin{align*}
  A=\prod_{j \in J : \langle \rho_j,
  h\rangle<0}a_j^{-\langle \rho_j, h\rangle}
  && \text{and} &&
                   B=\prod_{j \in J :  \langle \rho_j,
                   h\rangle>0}a_j^{\langle \rho_j, h\rangle}
\end{align*}
Noting that $\Ker \pi$ is generated by $(0,1)$ and $\Ker \pi^\prime$
by $(-h,1)$, it is easy to see that the trinomial in question is
homogeneous.  This also makes it clear that \eqref{eq:2}
holds. 

Finally we check that the trinomial induces the desired
qG-deformations. Choose orientation and coordinates such that
$\rho_i=(0,1)$, $\rho_{i+1}=(1,0)$ and $N=\ZZ^2+\frac1{n}(1,q)$. As
before, write $q=p-1$, $w=\hcf(n,p)$, $n=wr$, $p=wa$. It is easy
to see that with these choices $M=\bigl\{(u_1,u_2)\in\ZZ^2\mid u_1+qu_2\equiv 0\pmod{n}\bigr\}$, 
$h=(-r,-r)\in M$, and 
$f=\bigl(\frac{1}{w},-\frac{1}{w}\bigr)\in N$.  We analyze the family determined by \eqref{eq:1} in the toric charts on $\widetilde{X}$.  It suffices to consider the simplicial cone $\sigma$ in $\widetilde{N}$ generated by
the vectors
\begin{align*}
  \ve_0=\tilde{\rho}_x=\begin{pmatrix}\frac{1}{w}\\-\frac{1}{w}\\1\end{pmatrix} 
&&
   \ve_1=\tilde{\rho}_y=\begin{pmatrix}0\\0\\1\end{pmatrix}
&&
   \ve_2=\tilde{\rho}_z=\begin{pmatrix}0\\1\\-w\end{pmatrix}
\end{align*}
in $\widetilde{N}=N\oplus \ZZ$. The calculation:
\[
\frac{1}{n}
\begin{pmatrix}
  1 \\ q \\ 0
\end{pmatrix} =
\frac{1}{wr}
\begin{pmatrix}
  1 \\ wa-1 \\ 0
\end{pmatrix}
=\frac1{r}\ve_0-\frac1{r}\ve_1
+\frac{a}{r}\ve_2+\frac{wa}{r}\ve_1
\]
shows that the singularity in $\widetilde{X}$ corresponding to $\sigma$ is $\frac1{r}(1,wa-1,a)$, and that the trinomial~\eqref{eq:1} gives the expected
qG-deformation
\[
(xy+A z^w+B z^{w-r}=0) \subset \textstyle \frac1{r}(1, aw-1, a)
\]
where $A$ and $B$ are now units in the local ring at the singularity.
\end{proof}

\begin{proof}[Proof of Theorem~\ref{thm:1}]
  It follows from Lemma~\ref{lem:1} that the singularities of $X$ are
  exactly the R-contents of the singularities of the toric surface
  $X_P$, thus $X$ has locally qG-rigid singularities as claimed. By
  Lemma~\ref{lem:2}, if $P^\prime$ is mutation equivalent to $P$ then
  the toric surface $X_{P^\prime}$ is qG-deformation equivalent to
  $X_P$, and then a generic qG-deformation of $X_{P^\prime}$ is
  qG-deformation equivalent to a generic $qG$-deformation of
  $X_P$. Thus we get a (set-theoretic) map $\frak{P} \to
  \frak{F}$ as in the statement. The map is surjective by definition
  of the class TG.
\end{proof}

As a corollary, we can give a new, geometric proof that the singularity content of $P$ is invariant under
mutation. Let $X$ be a generic deformation of $X_P$.
Lemma~\ref{lem:1} implies that $X$ is locally qG-rigid and that the
multiset of singularities of $X$ is $\cB$. It is easy to see that
$m=e(X^0)$ is the homological Euler number of the smooth locus
   $X^0$ of $X$. Thus the singularity content of $P$ is a
   diffeomorphism invariant of $X$.  By Lemma~\ref{lem:2}, if
   $P^\prime$ is mutation equivalent to $P$ and $X^\prime$ is a
   generic 
   qG-deformation of $X_{P^\prime}$, then $X^\prime$ is a
   qG-deformation of $X$.  Lemma~\ref{lem:1} now implies that we can
   qG-deform $X$ to $X^\prime$ through locally qG-rigid surfaces,
   hence $X^\prime$ is diffeomorphic to $X$.  Thus the singularity
   content of $P'$ coincides with that of $P$.

\begin{figure}[htbp]
  \centering
  \captionsetup{width=0.78\textwidth}
  \begin{subfigure}[b]{0.2\textwidth}
    \centering
    \includegraphics[scale=0.6]{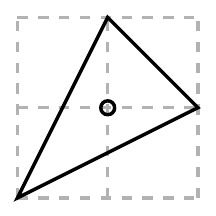}
    \subcaption*{$\PP^2$}
  \end{subfigure}
  \begin{subfigure}[b]{0.2\textwidth}
    \centering
    \includegraphics[scale=0.6]{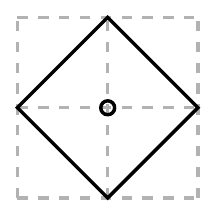}
    \caption*{$\PP^1 \times \PP^1$}
  \end{subfigure}
  \begin{subfigure}[b]{0.2\textwidth}
    \centering
    \includegraphics[scale=0.6]{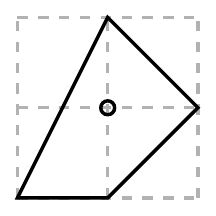}
    \caption*{$F_1$}
  \end{subfigure}
  \begin{subfigure}[b]{0.2\textwidth}
    \centering
    \includegraphics[scale=0.6]{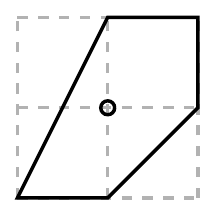}
    \caption*{$S^2_7$}
  \end{subfigure} \\[2.4ex]
  %----------------------------------------------------------------------
  \begin{subfigure}[b]{0.2\textwidth}
    \centering
    \includegraphics[scale=0.6]{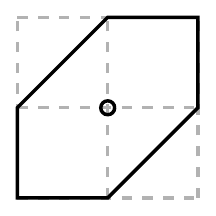}
    \subcaption*{$S^2_6$}
  \end{subfigure}
  \begin{subfigure}[b]{0.2\textwidth}
    \centering
    \includegraphics[scale=0.6]{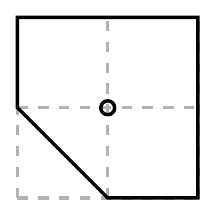}
    \caption*{$S^2_5$}
  \end{subfigure}
  \begin{subfigure}[b]{0.2\textwidth}
    \centering
    \includegraphics[scale=0.6]{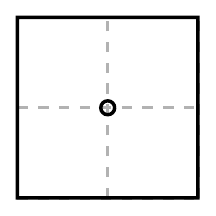}
    \caption*{$S^2_4$}
  \end{subfigure}
  \begin{subfigure}[b]{0.2\textwidth}
    \centering
    \includegraphics[scale=0.6]{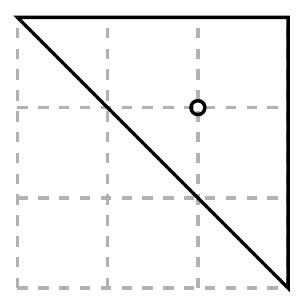}
    \caption*{$S^2_3$}
  \end{subfigure} \\[2.4ex]
  %----------------------------------------------------------------------
  \hspace{8mm}
  \begin{subfigure}[b]{0.246\textwidth}
    \centering
    \includegraphics[scale=0.6]{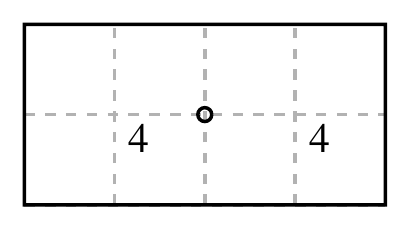}
    \subcaption*{$S^2_2$}
  \end{subfigure}  
  \hspace{2.24mm}
  \begin{subfigure}[b]{0.554\textwidth}
    \centering
    \includegraphics[scale=0.6]{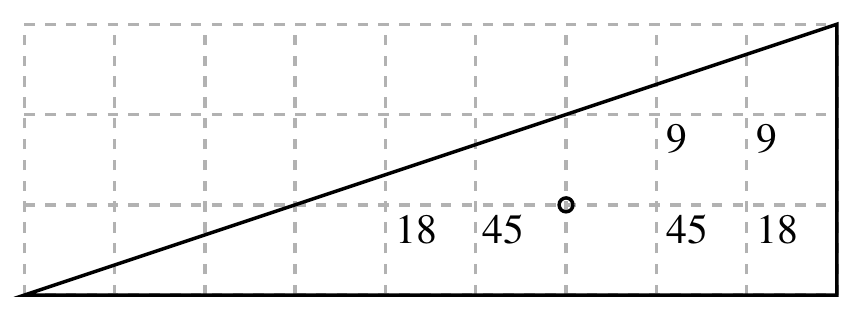}
    \subcaption*{$S^2_1$}
  \end{subfigure}  
  \caption{Representatives of the 10 mutation-equivalence classes of Fano polygons with singularity content $(n,\varnothing)$, labelled by the del~Pezzo surfaces to which they correspond under Conjecture~\ref{conjecture:A}.  Coefficients on interior lattice points specify maximally-mutable Laurent polynomials: see the main text.}
  \label{fig:smooth}
\end{figure}

\section*{The Evidence}
\label{sec:evidence}

\noindent We can prove our conjectures in the simplest cases, as we now explain.  

\subsection*{The Smooth Case}  It is well-known that there are precisely~10 deformation families of smooth del~Pezzo surfaces.  All of them are of class TG.  Fano polygons $P$ such that $X_P$ qG-deforms to a smooth del~Pezzo surface must have singularity content $(n,\varnothing)$ for some integer $n$.  Kasprzyk--Nill--Prince~\cite{KNP14} give an algorithm for classifying Fano polygons with given singularity content up to mutation, and thereby show that there are precisely~10 mutation-equivalence classes of Fano polygons with singularity content $(n,\varnothing)$ for some $n$.  These are illustrated in Figure~\ref{fig:smooth}.  Each such polygon supports a unique maximally-mutable Laurent polynomial~\cite{KT14}: these have zero as the coefficient of the constant monomial, coefficients of $(1+x)^k$ on each edge of length~$k$, and other coefficients as shown in Figure~\ref{fig:smooth}.  Combining the (known) classification of smooth del~Pezzo surfaces up to qG-deformation equivalence, the classification of the relevant polygons up to mutation-equivalence~\cite{KNP14}, and the computation of quantum periods $G_X$ for smooth del~Pezzo surfaces $X$~\cite[\S G]{QC105}, it is easy to see that Conjectures~\ref{conjecture:A},~\ref{conjecture:B},~\ref{conjecture:C}, and~\ref{conjecture:D} hold.

\subsection*{The Simplest Non-Smooth Case}

The simplest residual singularity is $\frac1{3}(1,1)$, so we consider now del~Pezzo surfaces with isolated singularities of this type only.  Such surfaces have been classified up to qG-deformation equivalence by Corti--Heuberger in~\cite{corti-heuberger14}:

\begin{thm}
  \label{thm:2}
  There are precisely $29$ qG-deformation families of del~Pezzo surfaces
  with $k\geq 1$ singular points of type $\frac1{3}(1,1)$, and precisely $26$ of
  these are of class TG.
\end{thm}

\noindent The classification result here can be derived from Fujita--Yasutake~\cite{FY14}. Corti--Heuberger also give an explicit construction of
  a generic surface in each family as a complete intersection in a toric orbifold or weighted flag variety, and determine exactly which of the families are of class TG.

Fano polygons $P$ such that $X_P$ qG-deforms to a singular del~Pezzo surface with only $\frac1{3}(1,1)$ singularities must have singularity content $\big(n,\{k \times \frac1{3}(1,1)\}\big)$ for some integers $n \geq 0$ and $k \geq 1$.  Such polygons have been classified up to mutation-equivalence by Kasprzyk--Nill--Prince in~\cite{KNP14}:

\begin{thm}
  \label{thm:3}
  There are precisely~$26$ mutation-equivalence classes of Fano
  polygons with singularity content $\bigl( n,\{k\times
  \frac1{3}(1,1)\}\bigr)$ for some integer~$n$ and some positive integer~$k$.
\end{thm}

\noindent The qG-deformation classes in Theorem~\ref{thm:2} and the mutation-equivalence classes in Theorem~\ref{thm:3} are in one-to-one correspondence, and Conjecture~\ref{conjecture:A} holds.  Kasprzyk--Tveiten have shown that each Fano polygon in Theorem~\ref{thm:3} supports a unique $k$-dimens\-ional family of maximally-mutable Laurent polynomials~\cite{KT14}; these have $T$-binomial edge coefficients.  Regarding Conjecture~\ref{conjecture:B}, one should bear in mind that computing the quantum period of orbifolds is a hard problem in Gromov--Witten theory: the constructions of Corti--Heuberger are at the limit of what can be treated using currently-available techniques.  Nonetheless Oneto--Petracci~\cite{OP14} have proved:

\begin{thm}
  \label{thm:4}
  Assuming natural generalizations of the Quantum Lefschetz Hyperplane Principle and the Abelian/non-Abelian Correspondence to the orbifold setting, for each of the $26$ families of class TG in Theorem~\ref{thm:2},
  there are Fano polygons $P$ and points $a_0\in L_P^T$ and $x_0\in
  H_X^{\text{ts}}$ such that:
  \[
  \widehat{G}_\sX (x_0,t)=\pi_P(a_0,t)
  \]
\end{thm}

\noindent This is a substantial step towards
Conjecture~\ref{conjecture:B} for this class of del~Pezzo surfaces.

Conjectures~\ref{conjecture:C} and~\ref{conjecture:D} also hold for
this class of del~Pezzo surfaces. In fact, we see from the classification that, in most cases, knowing the singularity content allows us to recover
the polygon. The four exceptions are: polygons $P_{12}$ and $P_{13}$ with
singularity content $\bigl(6, \{2\times \frac1{3}(1,1)\}\bigr)$, and
polygons $P_{21}$ and $P_{22}$ with singularity content $\bigl(5,
\{\frac1{3}(1,1)\}\bigr)$. The Laurent polynomials:
\begin{align*}
g_{12} & =x^{-3}y+6x^{-2}y+15x^{-1}y+20y+15xy+6x^2y+x^3y+ax^{-1}+bx+y^{-1}\\
g_{13  }&=x^{-1}y^{-1}+3y^{-1}+3xy^{-1}+x^2y^{-1}+3x^{-1}+a^\prime
           x+3x^{-1}y+b^\prime y+xy+x^{-1}y^2
\end{align*}
are the general maximally-mutable Laurent polynomials with Newton polygons $P_{12}$ and
$P_{13}$. A calculation shows that:
\begin{multline*}
\pi_{P_{12}}(a,b,t)=\pi_{g_{12}}(a,b,t)\\= 1+ (2ab + 40)t^2+(90a + 90b)t^3+
    (6a^2b^2 + 72a^2 + 480ab + 72b^2 + 5544)t^4 +\cdots  
\end{multline*}
and:
\begin{multline*}
\pi_{P_{13}}(a^\prime,b^\prime,t)=\pi_{g_{13}}(a^\prime,b^\prime,t)\\= 
1+(6a^\prime + 6b^\prime + 20)t^2+(6a^\prime b^\prime + 54a^\prime +
54b^\prime + 168)t^3+\\+
(90a^{\prime \,2} + 216a^\prime b^\prime + 900a^\prime + 90b^{\prime
  \, 2} + 900b^\prime + 2220)t^4+\cdots
\end{multline*}
It is immediate from these expressions that there is no affine-linear
isomorphism relating $a,b$ to $a^\prime, b^\prime$ that transforms
$\pi_{P_{12}}$ to $\pi_{P_{12}}$. A similar analysis establishes the
corresponding statement for $\pi_{P_{21}}$ and $\pi_{P_{22}}$.  This
proves Conjecture~C for del~Pezzo surfaces with only isolated singularities of type $\frac1{3}(1,1)$.

As for Conjecture~D for these surfaces, again, with the same four exceptions, the
qG-deformation type is determined by the degree and the basket of residual
singularities. For instance, the surface
$X_{P_{12}}$ deforms to a sextic in $\PP(1,1,3,3)$, and the surface
$X_{P_{13}}$ deforms to a general member $X$ of the family of hypersurfaces of type $L = (3,3)$
in the Fano simplicial toric variety $F$ with
weight matrix\footnote{The weight matrix defines an action of $(\Cstar)^2$ on $\CC^5$, and $F$ is the Fano GIT quotient of $\CC^5$ by this action. The line bundle $L$ over $F$ is defined by the character $(3,3)$ of $(\Cstar)^2$.}:
\[
\begin{array}{ccccc}
1 & 1 & 1 & 0 & 0  \\
0 & 0 & 1 & 1 & 3   
\end{array}
\]
It is easy to see, using the method of~\cite[Example~9]{CCIT2}, that these surfaces have different quantum
periods. This, together with a similar analysis of $X_{P_{21}}$ and
$X_{P_{22}}$, establishes Conjecture~D for del~Pezzo surfaces with only isolated singularities of type $\frac1{3}(1,1)$.

\subsection*{Classical and Quantum Invariants}

Let $P$ be a Fano polygon with basket of residual singularities
$\cB=\Bigl\{\textstyle\frac{1}{w_{0,j}r_j}(1,a_j w_{0,j}-1) :
\text{$j \in J$}\Bigr\}$. Consider a generic maximally-mutable Laurent
polynomial $f$ with Newton polygon $P$ and $T$-binomial edge
coefficients.  Regard $f$ as a map from $(\CC^\times)^2$ to
$\CC$. Tveiten has shown that a generic fibre
$\Gamma_\eta=f^{-1}(\eta)$ of $f$ is a curve of geometric genus
\[
g(\Gamma_\eta)=1+\sum_{j\in J} \frac{w_{o,j}(r_j-1)}{2}
\] and that
the monodromy endomorphism around $\infty$ acting on
$H_1(\Gamma_\eta, \ZZ)$ determines and is determined by the
singularity content of~$P$~\cite{Tveiten}.  One can think of the
singularity content as `classical information' which, as the examples
of $\PP^1 \times \PP^1$ and the Hirzebruch surface $F_1$ show, is
insufficient to determine the mutation-equivalence class of $P$;
Conjecture~\ref{conjecture:C} then suggests that the `quantum
information' required to determine this mutation-equivalence class is
the space $L_P^T$ of maximally-mutable Laurent polynomials with Newton
polytope $P$ and $T$-binomial edge coefficients.

\section*{Acknowledgements}
\label{sec:acknowledgements}

\noindent The program described here was the content of the PRAGMATIC
2013 Research School in Algebraic Geometry and Commutative Algebra
``Topics in Higher Dimensional Algebraic Geometry'' held in Catania,
Italy, in September~2013.  This paper was largely completed at the EMS School
``New Perspectives on the classification of Fano Manifolds'' held in
Udine, Italy, in September~2014.  We are very grateful to Alfio
Ragusa, Francesco Russo, and Giuseppe Zappal\`a, the organizers of the
PRAGMATIC school, and to Pietro De Poi and Francesco Zucconi, the
organisers of the EMS school, for creating a wonderful atmosphere for
us to work in.  We thank S\'andor Kov\'acs for advice on vanishing
theorems. 

\bibliographystyle{plain}
\bibliography{biblio}
\end{document}